\documentclass[smallextended]{svjour3}
\usepackage{xcolor} 
\usepackage{amsfonts}
\usepackage{amssymb}
\usepackage{amsfonts}
\usepackage{mathtools}
\usepackage{enumitem}
\usepackage{color}

\usepackage{float}
\usepackage{tikz}
\usetikzlibrary{positioning,calc,arrows,automata}

\newtheorem{rmk}{Remark}
\newtheorem{pro}{Proposition}
\newcommand{\one}{\mathbf{1}}
\newcommand{\rd}{\textcolor{red}}

\newcommand{\x}{\mathbf{x}}
\newcommand{\y}{\mathbf{y}}

\newcommand{\R}{\mathbb{R}}

\newcommand{\bl}{\begin{list}{ \ }{
\leftmargin=.325in}}
\begin{document}           

\title{Perron communicability and sensitivity of multilayer networks}

\author{Smahane El-Halouy \and Silvia Noschese \and Lothar Reichel}

\institute{S. El-Halouy \at
Department of Mathematical Sciences, Kent State University, Kent, OH 44242, and Laboratory
LAMAI, Faculty of Sciences and Technologies, Cadi Ayyad University, Marrakech, Morocco.\\
\email{elhalouysmahane@gmail.com}
\vskip2pt
S. Noschese \at
Dipartimento di Matematica ``Guido Castelnuovo'', SAPIENZA Universit\`a di Roma, P.le 
A. Moro, 2, I-00185 Roma, Italy. \\
\email{noschese@mat.uniroma1.it}
\vskip2pt
L. Reichel \at Department of Mathematical Sciences, Kent State University, Kent, OH 44242,
USA.\\
\email{reichel@math.kent.edu}
}

\dedication{Dedicated to Claude Brezinski on the occasion of his 80th birthday.}
\maketitle

\begin{abstract}
Modeling complex systems that consist of different types of objects leads to multilayer 
networks, where nodes in the different layers represent different kinds of objects. Nodes 
are connected by edges, which have positive weights. A multilayer network is 
associated with a supra-adjacency matrix. This paper investigates the sensitivity of the
communicability in a multilayer network to perturbations of the network by studying the 
sensitivity of the Perron root of the supra-adjacency matrix. Our analysis sheds light on
which edge weights to make larger to increase the communicability of the network, and
which edge weights can be made smaller or set to zero without affecting the 
communicability significantly.
\end{abstract}



\section{Introduction}\label{sec1}
Many complex systems can be modeled as networks. Informally, a network is a collection of 
objects, referred to as \emph{nodes} or \emph{vertices}, that are connected to each other 
in some fashion; the connections are referred to as \emph{edges}. The edges may be 
directed or undirected, and may be equipped with positive weights that correspond to their 
importance. The nature of the nodes, edges, and weights depends on the application. Some
modeling situations require more than one kind of nodes or more than one type of edges.

Multilayer networks are networks that consist of different kinds of edges and possibly 
different types of nodes. This kind of networks arise when one seeks to model a complex 
system that contains connections and objects with different properties. For instance, when
modeling train and bus connections in a country, the train routes and bus routes define 
edges with distinctive properties, and the train and bus stations may make up nodes with 
diverse properties. The connections between a train station and an adjacent bus station 
give rise to yet another kind of edges along which travelers walk. Edge weights may be 
chosen proportional to the number of travelers along an edge, proportional to the distance
between the nodes that the edge connects, or proportional to the cost of traveling along 
an edge. Whether it is meaningful to distinguish between different kinds of edges and 
nodes, and using edge weights, depends on the nature and purpose of the network model.

It is often of interest to determine the ease of communication between nodes in a network,
as well as how important a node is in some well-defined sense. Also, it is desirable to be 
able to assess the sensitivity of the measure of communication between the nodes to 
perturbations in the edge weights. For instance, if the nodes represent cities, and the 
edges represent roads between the cities, with edge weights proportional to the amount of 
traffic on each road, then one may be interested in which road(s) should be widened or made 
narrower to increase or reduce, respectively, communication in the network the most. The 
available data may be contaminated by measurement errors. We are then interested in how 
sensitive to errors in the data our choice of road(s) to widen or make narrower is. 

The investigation of the importance of nodes and edges, as well as the sensitivity of the
communicability within a network to changes in the edge weights of the network with only 
one kind of nodes and edges has received considerable attention in the literature; see, 
e.g., \cite{BRZ1,BRZ2,DJNR,DMR1,DMR2,Es,EH,Ne,NR} and references therein. Several of the 
techniques discussed evaluate the exponential of the adjacency matrix of the network, or 
the exponential of the adjacency matrix determined by the line graph associated with the 
given network. The present paper extends the communicability and sensitivity analysis in 
\cite{DJNR,NR} to multilayer networks. Since multilayer networks typically have a large 
number of nodes and edges, we focus on techniques that are well suited for large-scale 
networks.

We consider multilayer networks that are represented by graphs that share the same set of
vertices $V_N=\{1,2,\dots,N\}$ and have edges both within a layer and between layers. We
will simply refer to this kind of networks as \emph{multilayer networks}. Nice recent
discussions on multilayer networks are provided by Bergermann and Stoll \cite{BS},
Cipolla et al. \cite{CRZT}, and 
Tudisco et al. \cite{TAG}. De Domenico et al. \cite{DSOGA} describe how multilayer 
networks with $L$ layers can be modeled by a fourth order tensor and introduce a 
\emph{supra-adjacency matrix} $B\in\R^{NL\times NL}$ for the representation of such 
networks. In detail, let $A^{(\ell)}=[w_{ij}^{(\ell)}]_{i,j=1}^N\in\R^{N\times N}$ be the
non-negative adjacency matrix for the graph in layer $\ell$ for $\ell=1,2,\ldots,L$. Thus,
the entry $w_{i,j}^{(\ell)}\geq 0$ is the ``weight'' of the edge between node $i$ and node 
$j$ in layer $\ell$. If the graph is ``unweighted'', then all nonzero entries of 
$A^{(\ell)}$ are set to one. The matrix $B\in\R^{NL\times NL}$ is a block matrix with 
$N\times N$ blocks. The $\ell$th diagonal block is the adjacency matrix 
$A^{(\ell)}\in\R^{N\times N}$ for layer $\ell$, for $\ell=1,2,\ldots,L$; the off-diagonal 
$N\times N$ block in position $(\ell_1,\ell_2)$, with $1\leq \ell_1,\ell_2\leq L$ and 
$\ell_1\ne\ell_2$ represents the inter-layer connections between the layers $\ell_1$ and 
$\ell_2$; see Section \ref{sec3} for details.

We may consider $B$ an adjacency matrix for a monolayer network with $NL$ nodes,
and assume that $B$ is irreducible. This is equivalent to that the graph associated with 
$B$ is strongly connected; see, e.g., \cite{HJ}. Hence, 
the Perron-Frobenius theory applies, from which it 
follows that $B$ has a unique eigenvalue $\rho>0$ of largest magnitude (the 
\emph{Perron root}) and that the associated right and left eigenvectors, $\x$ and $\y$, 
respectively, can be normalized to be of unit Euclidean norm with all components positive.
These normalized eigenvectors are commonly referred to as the right and left  
\emph{Perron vectors}, respectively. Thus, 
\begin{equation}\label{evp}
B \x=\rho \x, \qquad \y^T B=\rho \y^T.
\end{equation}
We will assume throughout this paper that the Perron vectors $\x$ and $\y$ have been 
normalized in the stated manner.

Following \cite{DJNR}, we introduce the \emph{Perron communicability} in the multilayer network,
\begin{equation}\label{CPN}
C^{\rm PN}(B)=\exp_0(\rho)\one_{NL}^T\y\x^T\one_{NL}=\exp_0(\rho)\left(\sum_{j=1}^{NL}y_j\right)
\left(\sum_{j=1}^{NL}x_j\right),
\end{equation}
where
\[
\exp_0(t)=\exp(t)-1,\quad \x=[x_1,x_2,\dots,x_{NL}]^T,\quad \y=[y_1,y_2,\dots,y_{NL}]^T,
\]
and  $\one_{NL}\in\R^{NL}$ denotes the vector of all entries one. For a general adjacency 
matrix $B\in\R^{NL\times NL}$ associated with a monolayer network with $NL$ nodes, the 
above measure is analogous to, but fairly different from, the total network 
communicability 
\[
C^{\rm TN}(B)=\one_{NL}^T\exp(B)\one_{NL},
\]
introduced by Benzi and Klymko \cite{BK}. The latter is related to the ``size'' of the
matrix $\exp(B)$, while the measure \eqref{CPN} is determined by the Wilkinson perturbation 
$\y\x^T$ discussed in Section \ref{sec2}. The latter measure provides a worst-case
perturbation of the Perron root under a small perturbation of $B$. We use the modified 
exponential function $\exp_0(M)$ in \eqref{CPN} instead of the exponential, because the
Maclaurin series of $\exp(M)$ has no natural interpretation in the context of network 
modeling. We note that $C^{\rm PN}(M)$ is easy to apply and cheaper to compute than 
$C^{\rm TN}(M)$ and $C_0^{\rm TN}(M):=\one_{NL}^T\exp_0(M)\one_{NL}$ for monolayer 
networks with many nodes or layers, i.e., when $NL$ is large \cite{DJNR}. 

Due to the normalization of the Perron vectors $\x$ and $\y$ in \eqref{evp}, we have 
\[
1\leq\sum_{j=1}^{NL} x_j\leq\sqrt{NL},\qquad 1\leq\sum_{j=1}^{NL} y_j\leq\sqrt{NL}.
\]
Therefore, for the multilayer network associated with $B$, one has
\begin{equation}\label{CPNbd}
\exp_0(\rho)\leq C^{\rm PN}(B)\leq NL\exp_0(\rho).
\end{equation}
Typically, $\exp_0(\rho)\gg NL$. It then follows that the quantity $\exp_0(\rho)$ is a 
fairly accurate indicator of the Perron communicability of the graph represented by $B$ in
the sense that it suffices to consider $\exp_0(\rho)$ to determine whether the 
Perron communicability of a network is large or small. The right-hand side bound in 
\eqref{CPNbd} will be sharpened slightly in Proposition \ref{pro1} below.

Following the approach in  \cite{DSOGA}, we form the leading eigentensors 
$Y\in\R^{N\times L}$ and $X\in\R^{N\times L}$ for the multilayer network associated with
$B$ by reshaping the Perron vectors $\y$ and $\x$, respectively. Thus, the first column
of the matrix $Y$ is made up of the first $N$ components of the vector $\y$, the second 
column of $Y$ consists of the next $N$ components of the vector $\y$, etc. The \emph{joint 
eigenvector centrality} of node $i$ in layer $\ell$ is given by the entry in position 
$(i,\ell)$ of $Y$. The rows of $Y$ represent the \emph{eigenvector versatility} of the 
nodes. Moreover, the (scalar) \emph{versatility} of node $i$ is given by
\begin{equation}\label{nui}
\nu_i=(Y \one_L)_i,\;\;\;i=1,2,\dots,N.
\end{equation} 
The vector $\one_L$ may be 
replaced by some other vector in $\R^L$ with nonnegative entries if another weighting of 
the columns of $Y$ is desired. 

\begin{rmk}
The concepts of hub and authority communicability was introduced by Kleinberg \cite{Kl} 
for graphs that are defined by an adjacency matrix. An extension to multi-relational
networks that are based on tensors is described by Li et al. \cite{LNY}. We can define
analogous concepts for tensors by using the Perron communicability.
If we replace the matrix $B$ in \eqref{evp} by $BB^T$, then we obtain analogously to 
\eqref{CPN} the \emph{Perron hub communicability}
\[
C^{PN}(BB^T)=\exp_0(\rho_{\rm BB^T})\one_{NL}^T\x\x^T\one_{NL},
\]
where $\rho_{\rm BB^T}$ is the Perron root for $BB^T$ and $\x$ is the Perron vector
for $BB^T$. Similarly, if we replace the matrix $B$ in \eqref{evp} by $B^TB$, then we 
obtain the \emph{Perron authority communicability}
\[
C^{PN}(B^TB)=\exp_0(\rho_{\rm B^TB})\one_{NL}^T\x\x^T\one_{NL},
\]
where $\rho_{\rm B^TB}=\rho_{\rm BB^T}$ is the Perron root for $BB^T$ and $\x$ is the 
Perron vector for $BB^T$. 
\end{rmk}

We turn to special multilayer networks in which nodes in different layers are identified 
with each other. Thus, there are no edges between different nodes in different layers; 
the only edges that connect different layers are edges between a node and its copy in 
other layers. Hence, in the supra-adjacency matrix $B\in\R^{NL\times NL}$ all off-diagonal 
entries in all off-diagonal blocks are zero.

We will refer to these kinds of networks as \emph{multiplex networks}. They can be 
represented by a third-order tensor. The graph for layer $\ell$ is associated with the 
non-negative adjacency matrix $A^{(\ell)}\in\R^{N\times N}$, $\ell=1,2,\dots,L$, and a 
mode-1 unfolding of the third-order tensor that represents the network yields an 
$L$-vector of these adjacency matrices:
\begin{equation}\label{calA}
{\cal A}=[A^{(1)},A^{(2)},\dots,A^{(L)}]\in\R^{N\times NL}.
\end{equation}
The supra-adjacency matrix $B\in\R^{NL\times NL}$ for the multiplex network associated
with the matrix ${\cal A}$ in \eqref{calA} has the diagonal blocks $A^{(\ell)}$, $\ell=1,2,\dots,L$, and every 
$N\times N$ off-diagonal block is the identity matrix $I_N\in\R^{N\times N}$; see, e.g.,
\cite{DSOGA}. Hence, the coupling is diagonal and uniform. One may introduce a parameter 
$\gamma\geq 0$ that determines how strongly the layers influence each other. This yields 
the matrix 
\begin{equation}\label{suprad}
B:=B(\gamma)={\rm diag}[A^{(1)},A^{(2)},\dots,A^{(L)}]+
\gamma(\one_L\one_L^T\otimes I_N-I_{NL}),
\end{equation}
where $\otimes$ denotes the Kronecker product; see \cite{BS}.

Due to the potentially large sizes of the matrices $B$ in \eqref{evp} and \eqref{suprad},
one typically computes their right and left Perron vectors by an iterative method, which
only require the evaluation of matrix-vector products with the matrices $B$ and $B^T$.
Clearly, one does not have to store $B$, but only ${\cal A}$ in \eqref{calA} to evaluate 
matrix-vector products with the matrix $B$ in \eqref{suprad} and its transpose.

\begin{rmk}
If one is interested in the Perron hub or authority communicability of the network, then 
the matrices $A^{(\ell)}$ in \eqref{calA} should be replaced by $A^{(\ell)}(A^{(\ell)})^T$
or $(A^{(\ell)})^TA^{(\ell)}$, respectively, for $\ell=1,2,\dots,L$.
\end{rmk}

Following \cite[Definition 3.5]{TPM}, we introduce for future reference the 
$L$-dimensional vectors of the \emph{marginal layer $Y$-centralities} and the 
\emph{marginal layer $X$-centralities}
\begin{equation} \label{marginal}
\mathbf{c}_Y = Y^T\one_N\quad\mbox{and}\quad\mathbf{c}_X = X^T\one_N,
\end{equation}
respectively.

It is the purpose of the present paper to investigate the Perron network communicability 
of multilayer networks that can be represented by a supra-adjacency matrix 
$B\in\R^{NL\times NL}$, as well as the special case of multiplex networks that are
represented by the matrix ${\cal A}\in\R^{N\times NL}$ in \eqref{calA}. We also are 
interested in the sensitivity of the communicability to errors or changes in the entries 
of the supra-adjacency matrix $B$ and in the entries of the matrices $A^{(\ell)}$ in 
\eqref{calA} in the case of a multiplex network. The particular structure of $B$ in 
\eqref{suprad} for multiplex networks will be exploited. 

The organization of this paper is as follows. The Wilkinson perturbation for a 
supra-adjacency matrix is defined in Section \ref{sec2}.  This  perturbation forms the 
basis for our sensitivity analysis of multilayer networks.  Section \ref{sec1.5} discusses
some properties of the Perron and total network communicabilities. A sensitivity analysis 
for multilayer networks based on the Wilkinson perturbation is presented in Section 
\ref{sec3}.  Both Sections \ref{sec1.5}  and \ref{sec3} first discuss multilayer networks 
that can be defined by general supra-adjacency matrices, and subsequently describe 
simplifications that ensue for multiplex networks that can be defined by the supra-adjacency matrix 
$B$ in \eqref{suprad}. 
Section \ref{sec6} presents a few computed examples, and Section \ref{sec7} contains 
concluding remarks.

\section{Wilkinson perturbation for supra-adjacency matrices}\label{sec2}
Let $B\in\R^{NL\times NL}$ be the supra-adjacency matrix in \eqref{evp}. We assume that
$B$ is irreducible. Let $\rho>0$ be the Perron root of $B$, and let $\x$ and
$\y$ be the associated right and left normalized Perron vectors. Thus, all entries of $\x$
and $\y$ are positive, and $\|\x\|_2=\|\y\|_2=1$. Throughout this paper $\|\cdot\|_2$ 
denotes the Euclidean vector norm or the spectral matrix norm, and  $\|\cdot\|_F$ stands 
for the Frobenius norm. The vectors $\x$ and $\y$ are uniquely determined. 

Let $E\in\R^{NL\times NL}$ be a nonnegative matrix such that $\|E\|_2=1$, and let 
$\varepsilon>0$ be a small constant. Denote the Perron root of $B+\varepsilon E$ by 
$\rho+\delta\rho$. Then
\begin{equation}\label{pert}
\delta\rho=\varepsilon\frac{\y^TE\x}{\y^T\x}+O(\varepsilon^2);
\end{equation}
see \cite{MSN}. Moreover,
\begin{equation}\label{condrho}
\frac{\y^T E \x} {\y^T\x}=
\frac{|\y^T E \x|} {\y^T\x}\leq
\frac{\|\y\|_2\|E\|_2\|\x\|_2} {\y^T\x}=
\frac{1}{\cos\theta},
\end{equation}
where $\theta$ is the angle between $\x$ and $\y$. The quantity $1/\cos\theta$ is 
referred to as the \emph{condition number} of $\rho$ and denoted by $\kappa(\rho)$; see
Wilkinson \cite[Section\,2]{Wi}. Note that when $B$ is symmetric, we have $\x=\y$ and, 
hence, $\theta=0$. In this situation $\rho$ is well-conditioned. Equality in 
\eqref{condrho} is achieved for the \emph{Wilkinson perturbation} 
\begin{equation}\label{wilk}
E=\y\x^T\in\R^{NL \times NL},
\end{equation}
which we will refer to as $W$. For $E=W$, the perturbation \eqref{pert} of the Perron root
is $\delta\rho=\varepsilon\kappa(\rho)+O(\varepsilon^2)$. We observe that all the above 
statements hold true if everywhere the spectral norm is replaced by the Frobenius norm.

\section{Some properties of the Perron and total network communicabilities}\label{sec1.5}
This section discusses a few properties of the Perron communicability and how it relates 
to the total network communicability.

\begin{pro}\label{pro0}
\begin{eqnarray} \label{c0}
C^{\rm PN}(B)=\exp_0(\rho)\mathbf{c}_Y^T\mathbf{c}_X, 
\end{eqnarray}
where  $\mathbf{c}_X$ is the vector of the marginal layer 
$X$-centralities and $\mathbf{c}_Y$ is the vector of the marginal layer $Y$-centralities
in \eqref{marginal}.
\end{pro}
\begin{proof}
The proof follows from \eqref{CPN} by observing that
\[
\mathbf{1}_{NL}^T \mathbf{y}\mathbf{x}^T \mathbf{1}_{NL}=
\mathbf{1}_{N}^T Y X^T \mathbf{1}_{N} =\mathbf{c}_Y^T\mathbf{c}_X.
\]
\end{proof}

\begin{rmk}\label{symC}
When the network is undirected, one has according the definitions \eqref{marginal} that
$\mathbf{c}_X=\mathbf{c}_Y$, because $\x=\y$. This gives, by \eqref{c0}, the symmetric 
Perron communicability 
\begin{eqnarray*}
C^{\rm PN \,sym}(B)=\exp_0(\rho)\|\mathbf{c}_Y\|_2^2.
\end{eqnarray*}
\end{rmk}

\begin{pro}\label{pro1}
\[
C^{\rm PN}(B)\leq NL \exp_0(\rho)\cos \phi, 
\]
where $\phi$ is the angle between the vector $\mathbf{c}_Y$ of the marginal layer 
$Y$-centralities and the vector $\mathbf{c}_X$ of the marginal layer $X$-centralities
in \eqref{marginal}.
\end{pro}
\begin{proof}
One has
\[
\mathbf{c}_Y^T\mathbf{c}_X=\|\mathbf{c}_X \|_2\|\mathbf{c}_Y \|_2\cos \phi, 
\]
where $\phi$ is the angle between 
$\mathbf{c}_Y$ and $\mathbf{c}_X$.
Let $\|\cdot\|_1$ denote the vector $1$-norm. It is evident that
\[
\|\mathbf{c}_X \|_1=\sum_{j=1}^{NL} x_j =\|\mathbf{x} \|_1,\quad 
\|\mathbf{c}_Y \|_1=\sum_{j=1}^{NL} y_j =\|\mathbf{y} \|_1. 
\]
Since 
\[
\|\mathbf{c}_X\|_2\leq \|\mathbf{c}_X\|_1=\|\mathbf{x}\|_
1\leq \sqrt{NL}\|\mathbf{x}\|_2,\quad
\|\mathbf{c}_Y\|_2\leq \|\mathbf{c}_Y\|_1=\|\mathbf{y}\|_
1\leq \sqrt{NL}\|\mathbf{y}\|_2,
\]
we have the bound
\[
\|\mathbf{c}_X \|_2\|\mathbf{c}_Y \|_2\leq NL \|\mathbf{x}\|_2\|\mathbf{y}\|_2=NL,
\]
which gives the proof by using \eqref{c0}.
\end{proof} 

\begin{rmk}\label{symC2}
When the network is undirected, by Remark \ref{symC}, 
Proposition \ref{pro1} reads
\[
C^{\rm PN \,sym}(B)\leq NL \exp_0(\rho),
\]
which is the same bound as \eqref{CPNbd}.
\end{rmk}

Matrix function-based communicability measures have been generalized in \cite{BS} to the 
case of layer-coupled multiplex networks that can be represented by a supra-adjacency 
matrix $B$ of the form \eqref{suprad}, i.e.,  by ${\cal A}$ defined by \eqref{calA}.
Following the argument in \cite{DJNR}, assume that the Perron root $\rho$ of a 
supra-adjacency matrix $B$ of the form \eqref{suprad} is significantly larger than the 
magnitude of its the other eigenvalues. Then
\begin{equation*}\label{approx}
C_0^{\rm TN}({\cal A})\approx\kappa(\rho)C^{\rm PN}({\cal A}),
\end{equation*}
where $C_0^{\rm TN}({\cal A})=\one_{NL}^T\exp_0(B)\one_{NL}$ and $C^{\rm PN}({\cal A})$ 
refers to the Perron network communicability \eqref{CPN} when $B$ is of the form 
\eqref{suprad}. Thus, the multiplex total network communicability depends on the 
conditioning of the Perron root.

\begin{rmk}
It is  straightforward to see that if the network represented by the matrix $B$ of the 
form \eqref{suprad} is undirected, and the Perron root $\rho$ is significantly larger than
the magnitude of the other eigenvalues of $B$, then one has
\begin{equation*}\label{approxs}
C_0^{\rm TN \, sym}({\cal A})\approx C^{\rm PN\, sym}({\cal A}).
\end{equation*}
Indeed, the Perron vectors 
$\x$ and $\y$ coincide so that $\kappa(\rho)=1$.
\end{rmk}

\section{Multilayer network Perron root sensitivity}\label{sec3}
Let the supra-adjacency matrix $B\in\R^{NL\times NL}$ be associated with an $L$-layer
network as described above. Then an edge from node $i$ in layer $k$ to node $j$ in layer 
$\ell$, with $i,j\in\{1,2,\dots,N\}$, $i\ne j$, and $k,\ell\in\{1,2,\dots,L\}$, is 
associated with the $(i,j)$th entry $w_{ij}^{(k,\ell)}>0$ of the $(k,\ell)$th block of 
order $N\times N$ of the matrix $B$. 

Consider increasing the weight $w_{ij}^{(k,\ell)}$ of an existing edge by $\varepsilon>0$
or introducing a new edge from node $i$ in layer $k$ to node $j$ in layer $\ell$ with 
weight $\varepsilon>0$. This corresponds to perturbing the supra-adjacency matrix $B$ by
the matrix $\varepsilon E$, where the matrix $E\in\R^{NL\times NL}$ has entries zero 
everywhere, except for the entry one in position $(i,j)$ in the block $(k,\ell)$. It 
follows from \eqref{pert} that the impact on the Perron root of this perturbation is 
\[
\delta\rho=\varepsilon\kappa(\rho)\,{y_{N(k-1)+i}\,x_{N(\ell-1)+j}}+O(\varepsilon^2).
\]

The notion of \emph{multilayer network Perron root sensitivity} with respect to the 
direction $(i,k)\longrightarrow (j,\ell)$, defined by 
\begin{equation}\label{Smulti}
S^{\rm PR}_{i,\,j,\,k,\,\ell}(B):=\kappa(\rho)\,y_{N(k-1)+i }\,x_{N(\ell-1)+j },
\end{equation}
is helpful for determining which edge(s) to insert in, or remove from, a multilayer 
network.

\begin{rmk}
Notice that the largest entries of $\x$ and $\y$ are strictly smaller than $1$, hence the 
multilayer network Perron root sensitivity \eqref{Smulti} with respect of any direction is
less than $\kappa(\rho)$. Indeed, $\x$ and $\y$ are unit vectors with positive entries so 
that, if, e.g., $x_{N(\ell-1)+j}=1$,  this would imply that $x_k=0$ for all 
$k\ne N(\ell-1)+j$, which is not possible. 
\end{rmk}

We also introduce the \emph{multilayer network Perron root sensitivity matrix} associated
with $B$, denoted by $S^{\rm PR}(B)$, whose entries are given by the quantities 
$S^{\rm PR}_{i,\,j,\,k,\,\ell}(B)$. The following result holds true.

\begin{pro}\label{pro3}
The  multilayer Perron root  sensitivity matrix is given by
\begin{equation}\label{senmat}
S^{\rm PR}(B)= \kappa(\rho)W\in\R^{NL\times NL},
\end{equation}
where $W$ is the Wilkinson perturbation defined by \eqref{wilk}.
\end{pro}
\begin{proof}
The proof follows from \eqref{Smulti} by observing that
\[
S^{\rm PR}_{i,\,j,\,k,\,\ell}(B)=\kappa(\rho) W_{N(k-1)+i ,\,N(\ell-1)+j},
\]
with $W=\y\x^T$.
\end{proof}

\begin{rmk}\label{nS}
Notice that both the spectral norm and the Frobenius norm of the multilayer network Perron
root sensitivity matrix are equal to the condition number of the Perron root. Moreover, 
the Perron communicability \eqref{CPN} reads 
\[
C^{\rm PN}(B)=\frac{\exp_0(\rho)}{\kappa(\rho)} \one_{NL}^T S^{\rm PR}(B)\one_{NL}.
\]
\end{rmk}

\begin{rmk}\label{robust}
Following \cite[Eqs (2.1)-(2.2)]{NR}, the spectral impact of each existing edge in $B$
can be analyzed by means of the matrix
$$-\frac{1}{\rho}B\circ S^{\rm PR}(B)\in\R^{NL\times NL},$$
where $\circ$ denotes the Hadamard product.
\end{rmk}

The exponential of the spectral radius of the graph associated with $B$ often is a fairly 
accurate relative measure of the Perron network communicability of the graph; cf. 
\eqref{CPNbd}. If we would like to modify the graph by adding an edge that increases the 
Perron network communicability as much as possible, then we should choose the indices $i$,
$j$, $k$, and $\ell$ for the new edge so that
\[
x_{N(\ell-1)+j}=\max_{1\leq q\leq NL}x_q, \qquad 
y_{N(k-1)+i}=\max_{1\leq q\leq NL}y_q.
\]

We turn to the removal of an edge, with the aim of simplifying the graph without affecting 
the Perron network communicability much. We therefore would like to choose the indices 
$1\leq\hat{\imath},\hat{\jmath}\leq N$ and $1\leq\hat{k},\hat{\ell}\leq L$ such that 
$w_{\hat{\imath},\hat{\jmath}}^{(\hat{k},\hat{\ell})}$ is positive and 
\[
y_{N(\hat{k}-1)+\hat{\imath}}\,x_{N(\hat{\ell}-1)+\hat{\jmath}}=
\min_{\substack{1\leq i,j\leq N\\[1mm]1\leq k,\ell\leq L\\[0.5mm]w_{ij}^{(k,\ell)}>0}}
y_{N(k-1)+i}\,x_{N(\ell-1)+j}.
\]

A way to determine such an index quadruple $\{\hat{\imath},\hat{\jmath},\hat{k},\hat{\ell}\}$
is to first order the products $y_i x_j$, $1\leq i,j\leq NL$ in increasing order. This
yields a sequence of index pairs $\{i_q,j_q\}_{q=1}^{N^2L^2}$ such that 
\[
y_{i_q}x_{j_q}\leq y_{i_{q+1}}x_{j_{q+1}}\qquad\forall~1\leq q<N^2L^2.
\]
Then determine the first index pair $\{i_{\hat{q}},j_{\hat{q}}\}$ in this sequence such 
that $w_{\hat{\imath},\hat{\jmath}}^{(\hat{k},\hat{\ell})}>0$, where
\[
i_{\hat{q}}=N(\hat{k}-1)+\hat{\imath},\qquad j_{\hat{q}}=N(\hat{\ell}-1)+\hat{\jmath}
\]
with $1\leq\hat{\imath},\hat{\jmath}\leq N$ and $1\leq\hat{k},\hat{\ell}\leq L$.

We remark that the perturbation bound \eqref{pert} only is valid for $\varepsilon$ of 
small enough magnitude. Nevertheless, it is useful for choosing which edge(s) to remove to 
simplify a multilayer network. This is illustrated in Section \ref{sec6}. It may be 
desirable that the graph obtained after removing an edge is connected. The connectedness 
has to be verified separately. 

\begin{rmk}\label{sym}
Notice that when the network is undirected, it may be meaningful to require the 
perturbation of the network also be symmetric. Thus, instead of considering the network 
sensitivity \eqref{Smulti} with regard to the direction 
$(i,k)\longrightarrow (j,\ell)$, we investigate the sensitivity of the network with 
regard to perturbations in the directions $(i,k)\longrightarrow (j,\ell)$ and 
$(j,\ell)\longrightarrow (i,k)$. This results in the expression
\begin{eqnarray*}
S^{\rm PR\,sym}_{i,\,j,\,k,\,\ell}(B)&:=&
\kappa(\rho)\,(y_{N(k-1)+i}\,x_{N(\ell-1)+j}+y_{N(\ell-1)+j}\,x_{N(k-1)+i})\\
&=& 2x_{N(k-1)+i}\,x_{N(\ell-1)+j},
\end{eqnarray*}
where we have used that $\x=\y$. This expression is analogous to \eqref{Smulti}.
\end{rmk}

We conclude this section with a discussion on multiplex networks. In such a network, an 
edge from node $i$ to node $j$ in layer $\ell$, with $i,j\in\{1,2,\dots,N\}$, $i\ne j$, 
and $\ell\in\{1,2,\dots,L\}$ is associated with the entry $w_{ij}^{(\ell)}\geq 0$ of the 
adjacency matrix $A^{(\ell)}$. Increasing the weight $w_{ij}^{(\ell)}>0$ of an existing 
edge by $\varepsilon>0$, or introducing a new edge by setting a zero weight 
$w_{ij}^{(\ell)}$ to $\varepsilon>0$, means perturbing ${\cal A}$ in $\eqref{calA}$ by 
$\varepsilon{\cal P}$, where 
\begin{equation}\label{Pij}
{\cal P}=[O_N,\dots,O_N,P_{ij}^{(\ell)},O_N,\dots,O_N]\in\R^{N\times NL}\quad\mathrm{with} 
\quad P_{ij}^{(\ell)}=\mathbf{e}_i\mathbf{e}_j^T\in\R^{N\times N}.
\end{equation}
Here $O_N\in\R^{N\times N}$ denotes the zero matrix. The perturbation 
$\varepsilon{\cal P}$ corresponds to perturbing the supra-adjacency matrix $B$ by an
$NL\times NL$ block matrix with all null $N\times N$ blocks except for the $\ell$th 
diagonal block $A^{(\ell)}$ in which the $(i,j)$th entry is set equal to $\varepsilon$. 

Introduce the \emph{multiplex Perron root sensitivity} ${S}^{\rm PR}_{i,\,j,\,\ell}({\cal A})$
with respect to the direction $(i,j)$ in layer $\ell$, 
\[
{S}^{\rm PR}_{i,\,j,\,\ell}({\cal A}):=\kappa(\rho)\,{y_{N(\ell-1)+i }\,x_{N(\ell-1)+j }},
\]
which is analogous to  the quantity \eqref{Smulti} for more general multilayer networks. 
Thus, if ${\cal P}$ is defined by \eqref{Pij} and ${\cal A}$ by \eqref{calA}, one has from
\eqref{pert} that $\delta\rho\approx\varepsilon{S}^{\rm PR}_{i,\,j,\,\ell}({\cal A})$.
Analogously, consider reducing the $(i,j)$th entry $w_{ij}^{(\ell)}>0$ of the adjacency 
matrix $A^{(\ell)}$ by $\varepsilon$ and assume that $0<\varepsilon\ll 1$ and
$\varepsilon<w_{ij}^{(\ell)}$. Then the modified network associated with the tensor 
${\cal A}-\varepsilon{\cal P}$ is nonnegative and connected if the network associated with
${\cal A}$ has these properties. Then
$\delta\rho\approx\ -\varepsilon {S}^{\rm PR}_{i,\,j,\,\ell}({\cal A})$.

Moreover, as shown in Remark \ref{sym}, when considering an undirected multiplex network, 
we obtain the expression
\[
{S}^{\rm PR\,sym}_{i,\,j,\,\ell}({\cal A}):=2\,{x_{N(\ell-1)+i }\,x_{N(\ell-1)+j }}.
\]

Recall that the Perron root sensitivity matrix \eqref{senmat} for general multilayer 
networks depends on the Wilkinson perturbation $W\in{\mathbb R}^{NL\times NL}$ of the 
supra-adjacency $B$ as well as on the condition number $\kappa(\rho)$. By assuming that 
 $B$ is of the type in \eqref{suprad}, the results in the following section will lead 
to analogous properties of the {\it multiplex Perron root sensitivity matrix} 
${S}^{\rm PR}({\cal A})$, whose nonvanishing entries are given by the quantities 
${S}^{\rm PR}_{i,\,j,\,\ell}({\cal A})$.

\subsection{Exploiting the structure of multiplex networks}\label{sec4}
Consider the cone ${\cal D}$ of all nonnegative block-diagonal matrices in 
${\mathbb R}^{NL\times NL}$ with $L$ blocks in ${\mathbb R}^{N\times N}$ and let 
$M|_{\mathcal D}$ denote the matrix in ${\mathcal D}$ that is closest to a given matrix  
$M\in{\mathbb R}^{NL\times NL}$ with respect to the Frobenius norm. It is straightforward 
to verify that $M|_{\mathcal D}$ is obtained by replacing all the entries outside the 
block-diagonal structure by zero.

Let $E\in{\mathcal D}$ be such that $\|E\|_F=1$, and let 
$\varepsilon>0$ be a small constant. Then
\begin{equation}\label{condsrho}
\frac{\y^T E \x} {\y^T\x}=
\frac{|\y^T E \x|} {\y^T\x}\leq
\frac{\|\y\|_2\|\y\x^T|_{\mathcal D}\|_F\|\x\|_2} {\y^T\x}=
\frac{\|\y\x^T|_{\mathcal D}\|_F}{\y^T\x},
\end{equation}
with equality for the ${\mathcal D}$\emph{-structured analogue of the Wilkinson 
perturbation}
\begin{equation}\label{WD}
E=\frac{\y\x^T|_{\mathcal D}}{\|\y\x^T|_{\mathcal D}\|_F};
\end{equation}
see \cite{NP06}. The quantity 
\[
\frac{\|\y\x^T|_{\mathcal D}\|_F}{\y^T\x} = \kappa(\rho)\|\y\x^T|_{\mathcal D}\|_F 
\]
will be referred to as the ${\mathcal D}$\emph{-structured condition number} of $\rho$ and
denoted by $\kappa_{{\mathcal D}}(\rho)$. For $E$ in \eqref{WD}, the perturbation 
\eqref{pert} of the Perron root is 
$\delta\rho=\varepsilon\kappa_{{\mathcal D}}(\rho)+O(\varepsilon^2)$.

Thus, the ${\mathcal D}$-structured analogue of the Wilkinson perturbation is the maximal 
perturbation  for the Perron root $\rho$ of a supra-adjacency matrix of the type 
\eqref{suprad} induced by a ${\mathcal D}$-structured perturbation. The following result holds.

\begin{pro}\label{pro4}
The multiplex Perron root sensitivity matrix is given by
\[
{S}^{\rm PR}({\cal A})= \kappa(\rho)W|_{\mathcal D}\in\R^{NL\times NL}\,,
\]
where $W$ is the Wilkinson perturbation defined by \eqref{wilk} and ${\cal D}$ is the cone
of all nonnegative block-diagonal matrices in ${\mathbb R}^{NL\times NL}$ with $L$ blocks 
in ${\mathbb R}^{N\times N}$.
\end{pro}
\begin{proof}
In the multiplex associated with the matrix ${\cal A}$ in \eqref{calA} and represented by 
the matrix $B$ in the form \eqref{suprad}, the parameter $\gamma$ that yields the weight 
of the inter-layer edges, i.e., the influence of the layers on each other, is  determined 
a priori by the model. Thus, ${S}^{\rm PR}({\cal A})\in {\mathcal D}$, because admissible 
perturbations only affect intra-edges. Hence, one obtains from Proposition \ref{pro3} that
the multiplex Perron root sensitivity matrix consists of just the projection into 
${\mathcal D}$ of \eqref{senmat} obtained by replacing all the entries of $W$ outside the
block-diagonal structure by zero. This concludes the proof.
\end{proof}

Analogously to \eqref{senmat}, the multiplex Perron root sensitivity matrix is the product
of the maximal admissible perturbation and the relevant condition number of the Perron 
root. Thus, ${S}^{\rm PR}({\cal A})$ is given by the product of the 
${\mathcal D}$-structured condition number of $\rho$, $\kappa_{{\mathcal D}}(\rho)$, 
and the ${\mathcal D}$-structured analogue of the Wilkinson perturbation $W$:
\[
{S}^{\rm PR}({\cal A})= \kappa(\rho) \|W|_{\mathcal D}\|_F\frac{W|_{\mathcal D}}
{\|W|_{\mathcal D}\|_F}.
\]
Hence, the Frobenius norm of the multiplex  Perron root sensitivity matrix is equal to the
structured condition number $\kappa_{{\mathcal D}}(\rho)$ of the Perron root; see Remark 
\ref{nS} for the general case of a multilayer network.

The above argument quantitatively shows that the Perron communicability in multiplexes is 
less sensitive, both component-wise and norm-wise, than the Perron communicability in more
general multilayer networks.

\begin{rmk}\label{shbnd}
Following the argument in Remark \ref{nS}, we define the effective Perron communicability 
in a multiplex network,
\[
C^{\rm PN}({\cal A})=\frac{\exp_0(\rho)}{\kappa(\rho)} \one_{NL}^T S^{\rm PR}({\cal A})\one_{NL}.
\]
Moreover, observing that 
\[
\one_{NL}^T S^{\rm PR}({\cal A})\one_{NL}\leq NL\|S^{\rm PR}({\cal A})\|_F = 
NL \,\kappa_{\mathcal D}(\rho),
\]
we obtain the upper bound
\[
C^{\rm PN}({\cal A})\leq NL\exp_0(\rho)\|W|_{\mathcal D}\|_F,
\]
which is sharper than \eqref{CPNbd}.
\end{rmk}

We conclude this subsection by defining the multiplex Perron root sensitivity matrix 
\begin{equation*}
{S}^{\rm PR}({\cal A})= \kappa(\rho){\cal W}\,,
\end{equation*}
where
\begin{equation}\label{calW}
{\cal W}:=[W^{(1)},W^{(2)},\dots,W^{(L)}]\in\R^{N\times NL}.
\end{equation}
Here $W^{(\ell)}\in\R^{N\times N}$ is constructed by multiplying the $\ell$th column of
$Y$ by the $\ell$th row of $X^T$, for $\ell=1,2,\ldots,L$, where the matrices 
$X,Y\in\R^{N\times L}$ are determined by reshaping the right and left Perron vectors $\x$
and $\y$ of $B$; see Section \ref{sec1} for the definition of $X$ and $Y$.

\begin{rmk}
Analogously to Remark \ref{robust}, we note that the spectral impact of each existing 
edge in $\cal A$ can be studied by means of 
$$-\frac{1}{\rho}{\cal A}\circ {S}^{\rm PR}({\cal A});$$
cf.  \cite[Eqs (2.1)-(2.2)]{NR}. 
\end{rmk}

\subsection{Exploiting the sparsity structure of multiplex networks}
When considering perturbations of {\it existing} edges, we take into account the 
projection of the Wilkinson perturbation $W$ into the cone ${\cal S}$ of all matrices in
${\mathcal D}$ with the same sparsity structure as 
${\rm diag}[A^{(1)},A^{(2)}.\dots,A^{(L)}]$. The argument that lead to the 
{\it structured} results \eqref{condsrho} and \eqref{WD} holds true for any (further) 
sparsity structure of the matrix ${\rm diag}[A^{(1)},A^{(2)},\dots,A^{(L)}]$. Moreover, 
$\kappa_{\cal S}(\rho)\leq \kappa_{\cal D}(\rho)\leq\kappa(\rho)$. One has the following result 
for the {\it multiplex Perron root structured sensitivity matrix} 
${S}^{\rm PR\, struct}({\cal A})$,  whose nonvanishing entries are given by the quantities
${S}^{\rm PR}_{i,\,j,\,\ell}({\cal A})$ that correspond to the positive entries of $B$.

\begin{pro}
The multiplex Perron root structured sensitivity matrix is given by
\[
{S}^{\rm PR\, struct}({\cal A})= \kappa(\rho)W|_{\mathcal S}\in\R^{NL\times NL}\,\,,
\]
where $W$ is the Wilkinson perturbation defined by \eqref{wilk} and ${\cal S}$ is the cone
of all nonnegative block-diagonal matrices in ${\mathbb R}^{NL\times NL}$ with $L$ blocks 
in ${\mathbb R}^{N\times N}$ having the same sparsity structure as the diagonal block 
matrices of the supra-adjacency matrix $B$ in \eqref{suprad} that represents the multiplex.
\end{pro}
\begin{proof}
As for Proposition \ref{pro4}, the proof follows by observing that the multiplex Perron root 
structured sensitivity matrix $ {S}^{\rm PR\, struct}({\cal A})$ consists of the 
projection into ${\mathcal S}$ of $W$, because only perturbations of existing intra-edges 
are admissible. 
\end{proof}

We have the following component-wise and norm-wise inequalities:
\[
{S}^{\rm PR\, struct}({\cal A})\leq{S}^{\rm PR}({\cal A}),
\]
\[
\|{S}^{\rm PR\, struct}({\cal A})\|_F\leq\|{S}^{\rm PR}({\cal A})\|_F.
\]
\begin{rmk}

Following the argument in Remark \ref{shbnd}, we are in a position to introduce the 
notion of {\rm structured} Perron communicability in a multiplex network,
\[
C^{\rm PN\, struct}({\cal A})=\frac{\exp_0(\rho)}{\kappa(\rho)} \one_{NL}^T 
S^{\rm PN\, struct}({\cal A})\one_{NL},
\]
and obtain by using
\[
\one_{NL}^T S^{\rm PR\, struct}({\cal A})\one_{NL}\leq 
NL\|S^{\rm PR\, struct}({\cal A})\|_F = NL \,\kappa_{\mathcal S}(\rho)
\]
the sharper upper bound
\[
C^{\rm PN\, struct}({\cal A})\leq NL\exp_0(\rho)\|W|_{\mathcal S}\|_F\leq
NL\exp_0(\rho)\|W|_{\mathcal D}\|_F. 
\]
\end{rmk}

Finally, one may alternatively represent ${S}^{\rm PR\, struct}({\cal A})$ as
\begin{equation*}
{S}^{\rm PR\, struct}({\cal A})= \kappa(\rho){\cal W}||_{\mathcal S}\,,
\end{equation*}
where ${\cal W}||_{\mathcal S}$ is obtained from ${\cal W}$ in \eqref{calW}, by projecting
each matrix $W^{(\ell)}$ into the cone ${\cal S}^{(\ell)}$ of all nonnegative matrices in 
$\R^{N\times N}$ having the same sparsity structure as the matrix $A^{(\ell)}$, for 
$\ell=1,2,\ldots,L$.

\subsection{Symmetry patterns of multiplexes}
Let the network be represented by a symmetric supra-adjacency matrix $B$ of the type  
\eqref{suprad}. Applying the arguments in the preceding subsections to the cone of all the
symmetric matrices in ${\mathcal D}$ [all the symmetric matrices in ${\mathcal S}$] leads 
to the same structured analogue of the Wilkinson perturbation as $W|_{\mathcal D}$ 
[as $W|_{\mathcal S}$]. Indeed, as the network is undirected, the right and left Perron 
vectors coincide, so that the Wilkinson perturbation  $W=\y\x^T=\y\y^T$ is a symmetric 
matrix.

\section{Computed examples}\label{sec6}
This section presents some examples to illustrate the performance of the methods discussed
above. 
The computations were carried out 
using Matlab R2015b. 
The calculation of the Perron root and the left and right Perron vectors can easily be 
evaluated by using the Matlab functions {\sf eig} or {\sf eigs} for small to medium-sized 
network. For large-scale networks these quantities can be computed by the two-sided 
Arnoldi algorithm, which was introduced by Ruhe \cite{RUHE}, and has been improved by 
Zwaan and Hochstenbach \cite{ZH}.

\subsection{Example 1: A small synthetic multilayer network}
We construct a small directed unweighted general multilayer network with $N=4$ nodes in 
each layer and $L=3$ layers, illustrated in Figure \ref{fig1}. The network is represented
by a supra-adjacency matrix $B\in\R^{12\time 12}$, whose $4\times 4$ diagonal blocks are 
adjacency matrices that represent the graphs of each layer. The off-diagonal blocks 
represent edges that connect nodes in different layers. This results in a nonsymmetric matrix 
$B$, whose Perron root is $\rho(B)=2.3471$; the condition number of the Perron root is 
$\kappa(\rho(B))=1.0248$. 

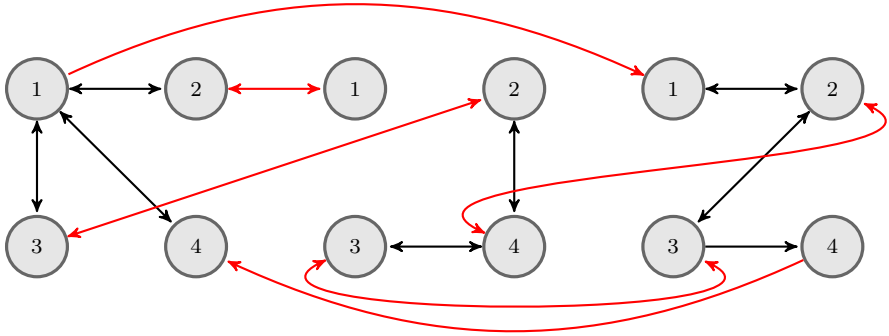
\begin{figure}[H]
\begin{center}
\begin{tikzpicture}[
 roundnode/.style={circle, draw=black!60, fill=black!10, very thick, minimum size=8mm},
 ->,>=stealth',shorten >=1pt,node distance=2.5cm,auto,thick]

\node[roundnode]      (node1)                                                   {$1$};
\node[roundnode]      (node2)       [right=1.25cm of node1]           {$2$};
\node[roundnode]      (node3)       [below=1.25cm of node2]         {$4$};
\node[roundnode]      (node4)       [below=1.25cm of node1]         {$3$};
\node[roundnode]      (node5)       [right=1.25cm of node2]          {$1$};
\node[roundnode]      (node6)       [right=1.25cm of node5]           {$2$};
\node[roundnode]      (node7)       [below=1.25cm of node6]         {$4$};
\node[roundnode]      (node8)       [below=1.25cm of node5]         {$3$};
\node[roundnode]      (node9)        [right=1.25cm of node6]           {$1$};
\node[roundnode]      (node10)       [right=1.25cm of node9]           {$2$};
\node[roundnode]      (node11)       [below=1.25cm of node10]         {$4$};
\node[roundnode]      (node12)       [below=1.25cm of node9]         {$3$};

\draw[<->] (node1) -- (node2);
\draw[<->] (node1) -- (node3);
\draw[<->] (node1) -- (node4);

\draw[<->] (node6) -- (node7);
\draw[<->] (node7) -- (node8);

\draw[<->] (node9) -- (node10);
\draw[<->] (node10) -- (node12);
\draw[->] (node12) -- (node11);
\rd{
\draw[<->] (node2) -- (node5);
}
\rd{
\draw[<->] (node4) -- (node6);
}

\rd{
\draw[<->] (node7) to [out=155,in=-25]  (node10);
}
\rd{
\draw[<->] (node8) to [out=-155,in=-25]  (node12);
}

\rd{
\draw[<-] (node9.155) to [out=155,in=25]  (node1.25);
}
\rd{
\draw[->] (node11.-155) to [out=-155,in=-25]  (node3.-25);
}
\end{tikzpicture}
\end{center}
\caption{Example 1: Layers are presented from left to right in the order $L=1$, $L=2$, and
$L=3$. The edges connecting nodes from same layer are marked in black. The edges 
connecting nodes from different layers are marked in red.}\label{fig1}
\end{figure}

Let $\varepsilon=0.3$ and let $W$ denote the Wilkinson perturbation \eqref{wilk}. Then
$\rho(B+\varepsilon W)=2.6512$. Thus, the perturbation $\varepsilon W$ of $B$ increases
the spectral radius by $0.3041$ as can be expected since $\varepsilon\kappa(\rho(B))=0.3074$. 
If we replace the matrix $W$ by the matrix of all ones, normalized to be of unit Frobenius
norm, then the spectral radius increases by only $0.2561$. Clearly, this is not an 
accurate estimate of the actual worst-case sensitivity of $\rho(B)$ to perturbations.

The largest Perron root sensitivity is $S^{\rm PR}_{2,4,3,2}(B)=0.2241$; cf. 
\eqref{Smulti}. This suggests that increasing the weight of the edge connecting node $2$
in layer $3$ and node $4$ in layer $2$ results in a relatively large change in the Perron 
root. 

In general, we expect the Perron root to increase more when introducing new eges or 
increasing edge weights that correspond to the largest entries of the Perron root 
sensitivity matrix $S^{\rm PR}(B)$ than when introducing randomly chosen edges or increasing 
randomly chosen edge weights. Table \ref{tab:table1} confirms this for Example 1.
Similarly, we expect a smaller decrease in the Perron root when decreasing the 
weights that correspond to the smallest entries of the matrix $S^{\rm PR}$ than when
decreasing random weights. Table \ref{tab:table2} confirms this for Example
1. 

The smallest entries of the matrix $S^{\rm PR}(B)$ also give the candidate edges to remove 
in order to simplify the network. However, we have to check the connectedness of the network 
after removal of an edge. Let $\hat B$ denote the supra-adjacency matrix obtained by 
removing the edge $(1,1)\longrightarrow (4,1)$, that connects node $1$ in layer $1$ and node $4$ in layer 
$1$. Then $\rho(\hat B)=2.3270$. Therefore, this removal decreases the Perron root only by an 
order of $10^{-2}$. Thus, the network represented by the supra-adjacency matrix $B$ can be 
simplified by removing the edge $(1,1)\longrightarrow (4,1)$ without a significant impact on the 
Perron network communicability. The graph obtained after removal of this edge is connected.

\begin{table}
\caption{Example 1: The four largest entries of the Perron root sensitivity matrix and 
Perron roots for the supra-adjacency matrix obtained by increasing/introducing the weights 
$w_{i,j}^{(k,\ell)}$, and Perron roots for supra-adjacency obtained by increasing/introducing the 
weight of random edges by $\varepsilon=0.3$.}
\label{tab:table1}
\begin{tabular}{ccccc}
$\{i,j,k,\ell\}$ & $S^{\rm PR}_{i,j,k,\ell}(B)$ & $\rho_{\rm new}$ & Random edges & 
$\tilde\rho_{\rm new}$ \\
   \noalign{\smallskip}\hline\noalign{\smallskip}
 $\{2,4,3,2\}$ & $0.2241$ & $2.4903$ & $\{1,3,3,3\}$ & $2.4041$ \\
 $\{4,3,2,3\}$ & $0.1725$ & $2.4592$ & $\{2,4,2,2\}$ & $2.4479$ \\
 $\{2,3,3,3\}$ & $0.1717$ & $2.4593$ & $\{1,3,1,1\}$ & $2.3975$ \\
 $\{3,4,2,2\}$ & $0.1694$ & $2.4627$ & $\{2,3,1,1\}$ & $2.3727$ \\
     \noalign{\smallskip}\hline\noalign{\smallskip} 
\end{tabular}
\end{table}

\begin{table}
\caption{Example 1: The four smallest entries of the Perron root sensitivity matrix and 
Perron roots for the supra-adjacency matrix obtained by decreasing the weights 
$w_{i,j}^{(k,\ell)}$, and Perron roots corresponding to decreasing the weight of random 
edges by $\varepsilon=0.3$.}
\label{tab:table2}
\begin{tabular}{ccccc}
$\{i,j,k,\ell\}$ & $S^{\rm PR}_{i,j,k,\ell} (B)$ & $\rho_{\rm new}$ & Random edges & 
$\tilde\rho_{\rm new}$ \\
   \noalign{\smallskip}\hline\noalign{\smallskip}
 $\{1,2,2,1\}$ & $0.0073$ & $2.3439$ & $\{3,3,2,3\}$ & $2.2822$ \\
 $\{3,4,3,3\}$ & $0.0211$ & $2.3407$ & $\{2,4,2,2\}$ & $2.2728$ \\
 $\{1,4,1,1\}$ & $0.0271$ & $2.3397$ & $\{1,2,3,3\}$ & $2.2935$ \\
 $\{1,2,1,1\}$ & $0.0331$ & $2.3332$ & $\{2,3,3,3\}$ & $2.2633$ \\
     \noalign{\smallskip}\hline\noalign{\smallskip} 
\end{tabular}
\end{table}

\subsection{Example 2: The ScotlandYard data set}
This example considers the Scotland Yard transportation network created by the authors of 
\cite{BS}. The network can be downloaded from \cite{B_repository}. It consists of $N=199$
nodes representing public transport stops in the city of London and $L=4$ layers that 
represent different modes of transportation: boat, underground, bus, and taxi. The edges 
are weighted and undirected. More precisely, the edges in the layer that represents 
travel by taxi all have weight one. A taxi ride is defined as a trip by taxi between two
adjacent nodes in the taxi layer; a taxi ride along $k$ edges is considered $k$ taxi 
rides. The weights of edges in the boat, underground, and bus layers are chosen to be 
equal to the minimal number of taxi rides required to travel between the same nodes. 

The Perron root of the supra-adjacency matrix $B$ is $\rho(B)=17.6055$, and its condition 
number is $\kappa(\rho(B))=1$. Let $\varepsilon=0.3$ and let $W$ be the Wilkinson
perturbation \eqref{wilk}. Then $\rho(B+\epsilon W)=17.9055$. Thus, the spectral radius 
increases by $0.3$. This can be expected since $\varepsilon\kappa(\rho(B))=0.3$. If we replace 
the matrix $W$ by the matrix of all ones, normalized to be of unit Frobenius norm, then 
the spectral radius increases by only $0.006$. This is not an accurate estimate of the 
actual worst-case sensitivity of $\rho(B)$ to perturbations. 

The largest entry of the Perron root sensitivity matrix is 
$S^{PR}_{89,67,2,2}(B)=0.2407$. Increasing the weight of the edge $e_{89,67,2,2}$ that 
connects the nodes $89$ and $67$ in layer $2$ typically results in a larger increase in 
the Perron root than when increasing a weight of a randomly chosen edge. For instance, 
when increasing the weight of the edge $e_{89,67,2,2}$ by $0.3$, the Perron root is
increased by $0.1458$; see Table \ref{tab:table3} for illustrations.

We also note that the Perron root $\rho(B)$ does not change significantly when setting the 
entry $(162,560)$ of $B$ to zero. This models the removal of the edge that connects node 
$162$ in layer $1$ to node $162$ in layer $3$ in the network. This edge corresponds to the
smallest entry of the Perron root sensitivity matrix 
$S^{\rm PR}_{162,162,1,3}(B)=3.2279\cdot10^{-15}$.

Now consider perturbations of existing edges. We compute the multiplex Perron root 
structured sensitivity matrix $S^{\rm PR\,struct}({\cal A})$ and compare the changes in the Perron root when increasing 
the weights of existing edges according to the largest entries of $S^{\rm PR\,struct}({\cal A})$ and
increasing the weights of randomly chosen existing edges. This is illustrated by Table 
\ref{tab:table4}. As expected, the Perron root changed the most when considering edges 
associated with a large entry in the matrix $S^{\rm PR\,struct}({\cal A})$. 

Finally, we note that the Perron root of the network is not very sensitive to removal of
edges that correspond to the smallest entries of the matrix $S^{\rm PR\,struct}({\cal A})$; see Table
\ref{tab:table5}.

\begin{table}
\caption{Example 2: The three largest entries of the Perron root sensitivity matrix and 
Perron roots for the supra-adjacency matrix obtained by increasing the weights 
$w_{i,j}^{(k,\ell)}$ by $\varepsilon=0.3$, and Perron roots corresponding to same increase 
for random edges.}
\label{tab:table3}
\begin{tabular}{ccccc}
$\{i,j,k,\ell\}$ & $S^{\rm PR}_{i,j,k,\ell} (B)$ & $\rho_{\rm new}$ & Random edges & 
$\tilde\rho_{\rm new}$ \\
   \noalign{\smallskip}\hline\noalign{\smallskip}
 $\{89,67,2,2\}$ & $0.2407$ & $17.7513$ & $\{103,40,4,4\}$ & $17.6055$ \\
 $\{89,13,2,2\}$ & $0.2041$ & $17.7299$ & $\{7,188,3,3\}$ & $17.6055$ \\
 $\{13,67,2,2\}$ & $0.1821$ & $17.7161$ & $\{174,162,3,3\}$ & $17.6055$ \\
     \noalign{\smallskip}\hline\noalign{\smallskip} 
\end{tabular}
\end{table}

\begin{table}
\caption{Example 2: Sensitivity of the Perron root to structured increase of weights by $\varepsilon=0.3$.}
\label{tab:table4}
\begin{tabular}{ccccc}
$\{i,j,\ell\}$ & $S^{\rm PR\,struct}_{i,j,\ell} ({\cal A})$ & $\rho_{\rm new}$ & Random edges & 
$\tilde\rho_{\rm new}$ \\
   \noalign{\smallskip}\hline\noalign{\smallskip}
 $\{89,67,2\}$ & $0.2407$ & $17.7513$ & $\{13,52,3\}$ & $17.6057$ \\
 $\{89,13,2\}$ & $0.2041$ & $17.7299$ & $\{74,46,2\}$ & $17.6085$ \\
 $\{67,13,2\}$ & $0.1821$ & $17.7161$ & $\{108,117,4\}$ & $17.6055$ \\
 $\{67,111,2\}$ & $0.1315$ & $17.6861$ & $\{98,97,4\}$ & $17.6055$ \\
 $\{89,140,2\}$ & $0.1309$ & $17.6858$ & $\{158,142,4\}$ & $17.6055$ \\
     \noalign{\smallskip}\hline\noalign{\smallskip} 
\end{tabular}
\end{table}

\begin{table}
\caption{Example 2: Sensitivity of the Perron root to removal of edges.}
\label{tab:table5}
\begin{tabular}{ccccc}
$\{i,j,\ell\}$ & $S^{\rm PR\,struct}_{i,j,\ell} ({\cal A})$ & $\rho_{\rm new}$ & Random edges & 
$\tilde\rho_{\rm new}$ \\
   \noalign{\smallskip}\hline\noalign{\smallskip}
 $\{175,162,4\}$ & $2.0199\cdot10^{-12}$ & $17.6055$ & $\{67,111,2\}$ & $16.6289$ \\
 $\{7,6,,4\}$ & $4.6646\cdot10^{-12}$ & $17.6055$ & $\{102,103,4\}$ & $17.6055$ \\
 $\{30,17,4\}$ & $4.7102\cdot10^{-12}$ & $17.6055$ & $\{11,100,3\}$ & $17.6054$ \\
 $\{17,7,4\}$ & $4.7552\cdot10^{-12}$ & $17.6055$ & $\{79,46,2\}$ & $17.4191$ \\
     \noalign{\smallskip}\hline\noalign{\smallskip} 
\end{tabular}
\end{table}

\subsection{Example 3: The European airlines data set}
The European airlines data set consists of $450$ nodes that represent European airports 
and has $37$ layers that represent different airlines operating in Europe. Each edge
represents a flight between airports. There are $3588$ edges. The network can be 
represented by a supra-adjacency matrix $B$ (\ref{suprad}), where the block-diagonal 
matrices contain ones if an airline offers a flight between the two corresponding 
airports, and zeros otherwise. Each off-diagonal block is the identity matrix; this 
reflects the effort required to change airlines for connecting flights. The network can be
downloaded from \cite{B_repository}.

Similarly as Taylor et al. \cite{TPM}, we only include $N=417$ nodes from the the largest 
connected component of the network. This component defines the supra-adjacency matrix $B$.
Its largest eigenvalue is $\rho(B)=38.3714$ and $\kappa(\rho(B))=1$. Let $\varepsilon=0.3$ 
and let $W$ be the Wilkinson perturbation. Then $\rho(B+\varepsilon W)=38.6714$. Thus, the
spectral radius increases by $0.3$ as expected since $\varepsilon\kappa(\rho(B))=0.3$. 

If we replace the matrix $W$ by the matrix of all ones, normalized to be of unit Frobenius
norm, then the spectral radius increases by only $0.091$. 

The smallest entry of the Perron root sensitivity matrix is 
$S^{\rm PR}_{202,202,31,28}(B)=5.1845\cdot 10^{-13}$. This suggests that the cost of changing 
from the Czech airline to the Niki airline at Valan Airport can be avoided without influencing 
the communicability of the network.

The two largest entries of the Perron root sensitivity matrix are 
$S^{PR}_{38,2,1,1}(B)=0.0040$ and $S^{PR}_{157,2,1,1}(B)=0.0034$. This indicates that the 
Perron root may be increased the most by increasing the number of flights operated by the
Lufthansa airline between the Munich and Frankfurt Am Main airports and between 
D\"usseldorf and Frankfurt Am Main airports.

Finally, we consider structured perturbations. Table \ref{tab:table6} shows important 
changes in the Perron root when increasing the weights $w^\ell_{i,j}$ corresponding to the
largest entries of the multiplex Perron root structured sensitivity matrix $S^{\rm PR\,struct}({\cal A})$ compared 
to increasing weights of random existing edges by $\varepsilon=0.3$. On the other hand, 
removing random edges decreases the Perron root more than removing edges that correspond 
to the smallest entries of $ S^{\rm PR\,struct}({\cal A})$; see Table $\ref{tab:table7}$.

We conclude that the Perron communicability of the European airlines network is not so 
sensitive to removing flights operated by Wideroe Airlines between several airports. 
Meanwhile, increasing the number of flights operated by the Lufthansa airline would 
increase the communicability of the network significantly.

\begin{table}
\caption{Example 3: Sensitivity of the Perron root to structured increase of weights $\varepsilon=0.3$.}
\label{tab:table6}
\begin{tabular}{ccccc}
$\{i,j,\ell\}$ & $S^{\rm PR\,struct}_{i,j,\ell} ({\cal A})$ & $\rho_{\rm new}$ & Random edges & 
$\tilde\rho_{\rm new}$ \\
   \noalign{\smallskip}\hline\noalign{\smallskip}
 $\{2,38,1\}$ & $0.0040$ & $38.3738$ & $\{31,157,13\}$ & $38.3719$ \\
 $\{2,157,1\}$ & $0.0034$ & $38.3734$ & $\{246,12,2\}$ & $38.3717$ \\
 $\{157,38,1\}$ & $0.0033$ & $38.3734$ & $\{32,164,2\}$ & $38.3715$ \\
 $\{50,2,1\}$ & $0.0026$ & $38.3730$ & $\{27,64,28\}$ & $38.3719$ \\
 $\{50,38,1\}$ & $0.0026$ & $38.3729$ & $\{107,78,2\}$ & $38.3718$ \\
     \noalign{\smallskip}\hline\noalign{\smallskip} 
\end{tabular}
\end{table}

\begin{table}
\caption{Example 3: Sensitivity of the Perron root to structured removal of edges.}
\label{tab:table7}
\begin{tabular}{ccccc}
$\{i,j,l\}$ & $S^{\rm PR\,struct}_{i,j,\ell} ({\cal A})$ & $\rho_{\rm new}$ & Random edges & 
$\tilde\rho_{\rm new}$ \\
   \noalign{\smallskip}\hline\noalign{\smallskip}
 $\{350,316,35\}$ & $1.5058\cdot10^{-11}$ & $38.3714$ & $\{61,2,1\}$ & $38.3689$ \\
 $\{202,144,35\}$ & $1.5300\cdot10^{-11}$ & $38.3714$ & $\{64,170,6\}$ & $38.3695$ \\
 $\{316,144,35\}$ & $1.6606\cdot10^{-11}$ & $38.3714$ & $\{237,15,4\}$ & $38.3697$ \\
  $\{202,270,35\}$ & $1.4789\cdot10^{-11}$ & $38.3714$ & $\{71,80,4\}$ & $38.3704$ \\
 $\{350,144,35\}$ & $3.6032\cdot10^{-11}$ & $38.3714$ & $\{26,15,9\}$ & $38.3691$ \\

     \noalign{\smallskip}\hline\noalign{\smallskip} 
\end{tabular}
\end{table}

\subsection{Example 4: General multilayer network}
We consider an example of a general multilayer network, where interactions are allowed 
between different nodes in different layers. The network has $160$ nodes, $6$ layers, and
$148$ edges that may be directed. The network can be downloaded from 
\texttt{https://github.com/wjj0301/Multiplex-Networks}.\

The Perron root of the supra-adjacency matrix $B$ associated with the network, and its 
condition number are $\rho(B)=8.1324$ and $\kappa(\rho(B))=1.3277$, respectively. Let 
$\varepsilon=0.3$ and let $W$ be the Wilkinson perturbation. Then the Perron root of 
$B+\varepsilon W$ is $0.3990$ larger than $\rho(B)$. This can be expected since 
$\varepsilon\kappa(\rho(B))=0.3983$. The largest entry of the Perron root sensitivity matrix 
is $S^{\rm PR}_{6,24,1,1}(B)=0.3389$. Increasing the weight of the edge connecting node 
$6$ and node $24$ in layer $1$ by $0.3$ increases the Perron root by $0.0998$.

We used $\varepsilon=0.3$ in all computed examples. The conclusions drawn would have been
the same if instead $\varepsilon=0.1$ were used. 



\section{Conclusion}\label{sec7}
This paper investigates the communicability of multilayer networks by introducing the 
concept of Perron communicability for this kind of networks. The communicability is 
measured by the Perron root of the supra-adjacency matrix associated with the network.
The Perron vectors of this matrix help to determine which edge weights to increase or 
reduce in order to increase or reduce, respectively, the Perron communicability the 
most. Our analysis also addresses the question of which edges can be removed without 
changing the Perron communicability much. 

%
%

\end{document}